\documentclass[a4paper, 12pt, reqno]{amsart}
\allowdisplaybreaks
\usepackage{amsmath,amssymb,amsthm}
\usepackage{mathtools}
\usepackage[left=2.7cm, right=2.7cm, top=3cm, bottom=3cm]{geometry}
\usepackage[colorlinks = true, urlcolor = blue, linkcolor = blue,
            citecolor = blue]{hyperref}
\newif\ifmtpro 
\usepackage{euscript}
\usepackage[shortlabels]{enumitem}
\usepackage{xcolor}
\ifmtpro

   \usepackage[scaled=0.85]{helvet}
   \usepackage[subscriptcorrection]{mtpro}
   \usepackage[mtpcal]{mtpb}
\fi
\usepackage{microtype}    
\DeclareMathAlphabet{\boocal}{U}{BOONDOX-cal}{m}{n}
\newtheorem{theorem}{Theorem}
\newtheorem{lemma}{Lemma}
\newtheorem{corollary}{Corollary}

\theoremstyle{definition}

\newtheorem{remark}{Remark}

\newcommand{\E}{\mathbb E} 
\newcommand{\Var}{\text{Var}}
\newcommand{\Z}{\mathbb Z}
\newcommand{\N}{\mathbb N}

\newcommand{\R}{\mathbb R}
\newcommand{\Pb}{\mathbb P}
\newcommand{\sF}{\mathfrak F}
\newcommand{\sL}{\mathfrak L}
\newcommand{\D}{\mathfrak{D}} 

\newcommand{\TsD}{\textup{T}\D}

\newcommand{\JF}[1]{\mathit{JF}(#1)}
\ifmtpro
   \newcommand{\lstp}{\mathscr{l}}
   \newcommand{\rstp}{\mathscr{\altr}}
   \newcommand{\CR}{\mathscr{R}(\Phi)}
   \newcommand{\Eq}{\boocal{E}} 
\else
   \newcommand{\lstp}{l}
   \newcommand{\rstp}{r}
   \newcommand{\CR}{\mathcal{R}(\Phi)}
   \newcommand{\Eq}{\boocal{E}}
\fi
\begin{document}
\title{Two repelling random walks on $\Z$}
\author[F. P. A. Prado]{Fernado P. A. Prado}
\author[C. F. Coletti]{Cristian F. Coletti}
\address[C. F. Coletti]{Centro de Matem\'atica, Computa\c{c}\~ao e
       Cogni\c{c}\~ao, UFABC, Avenida dos Estados, 5001, Santo Andr\'e,
       S\~ao Paulo, Brasil}
\email{cristian.coletti@ufabc.br}
\thanks{C. F. Coletti was partially supported by grant
  \#17/10555-0 S\~ao Paulo Research Foundation (FAPESP)}
\address[F. P. A. Prado and
         R. A. Rosales]{Departameto de Computa\c{c}\~ao e
           Matem\'atica, Universidade de S\~ao Paulo, 
	   Avenida Bandeirantes 3900, Ribeir\~ao Preto, S\~ao Paulo, 
	   14040-901, Brasil}
\email{feprado@usp.br, rrosales@usp.br}
\author[R. A. Rosales]{Rafael A. Rosales}

\subjclass[2010]{Primary 60K35, Secondary 37C10}
\keywords{reinforced random walk, recurrence, transience, stochastic
  approximations, stability}
\date{22/17/2021}

\begin{abstract} We consider two interacting random walks on $\Z$ such
that the transition probability of one walk in one direction decreases
exponentially with the number of transitions of the other walk in that
direction. The joint process may thus be seen as two random walks
reinforced to repel each other. The strength of the repulsion is
further modulated in our model by a parameter $\beta \geq 0$. When
$\beta = 0$ both processes are independent symmetric random walks on
$\Z$, and hence recurrent. We show that both random walks are further 
recurrent if $\beta \in (0,1]$. We also show that these processes are
transient and diverge in opposite directions if $\beta > 2$.  The
case $\beta \in (1,2]$ remains widely open. Our results are obtained
by considering the dynamical system approach to stochastic
approximations.
\end{abstract}

\maketitle 

\section{Introduction}
We are concerned with the recurrence properties of two repelling
random walks $\{S^i_n$; $i=1,2, n\geq 0\}$ taking values on $\Z$ in
which the repulsion is determined by the full previous history of the
joint process. Formally, assume that $S_0^i,\ldots S_{n_0}^i \in \Z$
are known for given but arbitrary $n_0 \geq 1$, and let $\sF_n =
\sigma(\{S_k^1, S_k^2 : 0 \leq k \leq n\})$ be the natural filtration
generated by both walks. The transition probability for each process
is defined as
\begin{equation}\label{eqn:transProb}
\Pb\big(S_{n+1}^i = S_n^i + 1 \,  \big{|}  \,  \sF_n \big)
=
\psi\big((S_{n}^j - S_{0}^j)/n\big)
=
1 - \Pb\big(S_{n+1}^i = S_n^i - 1 \, \big{|} \, \sF_n \big),
\end{equation}
with $i=1,2$, $j=3-i$, $n \geq n_0$, and $\psi: [-1,1] \to [0,1]$,
defined by
\begin{equation}\label{eqn:psi}
  \psi(y) = \frac{1}{1 + \exp(\beta y)},
  \quad\beta \geq 0.
\end{equation}

When $\beta = 0$, then $\psi(y) = \frac{1}{2}$ for all $y \in [-1,1]$
and both $S^1_n$ and $S^2_n$ form two independent simple random walks
on $\Z$. To analyse the behaviour for $\beta > 0$, note that the
quantity $y = (S^j_n - S^j_0)/n$ represents the difference between the
proportions of times the $j$-th walk made a right and a left
transition up to time $n$. Thus, if $y$ is positive, then $S_{n+1}^i$
transits with highest probability $1 - \psi(y) > \frac12$ to the left. By
contrast if $y < 0$, that is, if $S_n^j$ has moved more to the left
than to the right, then $S_{n+1}^i$ moves to right with highest
probability $\psi(y) > \frac12$. It is worth mentioning that $\psi$
satisfies the following symmetry relation $\psi(-y) = 1-\psi(y)$, and
hence it is not biased in any direction, left or right. The parameter
$\beta$ strengthens the repulsion between the walks: the larger the
value of $\beta$, the higher is the probability each walk goes in the
direction less transited by the other walk.  For given arbitrary
initial conditions, the coordination of the walks towards a limiting
direction, if any, is far from trivial.

We regard a walk $S_n^i$ as recurrent (transient) if every vertex of
$\Z$ is visited by $S_n^i$ infinitely (only finitely) many times
almost surely. Our main results are stated as follows.

\begin{theorem}\label{th:trans} If $\beta > 2$, both random walks
$S^1_n$ and $S^2_n$ are transient and
\[
  \lim_{n\to\infty} S^1_n = 
  -\lim_{n\to\infty} S^2_n = \pm \infty \quad \textup{a.s.}
\]
\end{theorem}

\begin{theorem}\label{th:rec}
If $\beta \in [0, 1]$, then both $S_n^1$ and $S_n^2$ are recurrent.
\end{theorem}

\begin{remark}\label{conj:recurrence} The case $\beta = 0$ is
trivial. Indeed, when $\beta=0$, both $S_n^1$ and $S_n^2$ are two
independent simple symmetric random walks and hence recurrent. The
case $\beta \in (1, 2]$ remains widely open. The problem that arises
in this case is mentioned the end of this article in
Remark~\ref{remark:speed}.
\end{remark}

According to (\ref{eqn:transProb}) and (\ref{eqn:psi}), the
probability of a transition in a given direction decreases with the
number of previous transitions made by the opponent walk in that
direction. This allows to recognise the process studied throughout as
being formed by two interacting reinforced random walks, namely one in
which the reinforcement is set by repulsive behaviour of each
walk. Self-attracting reinforced random walks were formally introduced
in an unpublished paper by D. Coppersmith and P. Diaconis and have
since been the subject of intense research, see for instance
\cite{D90}, \cite{P92}, \cite{B97}, \cite{V01}, \cite{T04},
\cite{MR09}, \cite{ACK14}, and \cite{CT17}. Self-repelling walks, have
also deserved some attention, see \cite{T95}, \cite{T01} and
references therein. The recurrence properties of self-attracting walks
have been considered among others by \cite{S06}, \cite{MR09},
\cite{S14} and \cite{CK14}.  With the exception of \cite{C2014}, there
are relatively few studies of interacting vertex reinforced random
walks with `competition' or `cooperation'. \cite{C2014} considers two
random walks that compete for the vertices of finite complete graphs
and focuses on the asymptotic properties of the overlap of their
vertex occupation measures.

In this article, we study the recurrence properties of $S_n^i$,
$i=1,2$, by analysing the proportions of times each walk $i$ makes a
left and a right transition up to time $n$. To do so, we identify the
vector of empirical measures defined by these proportions with a
stochastic approximation process.  The latter have been quite
effective while dealing with several reinforced processes such as
vertex reinforced walks and generalized P\'olya urns, see \cite{P07}
for a survey and further references.  More precisely, we study the
asymptotic behaviour of the involved stochastic approximation by
considering the dynamical system approach described in \cite{B96} and
\cite{B99}. The rest of this article is organised as
follows. Section~\ref{sec:DS} shows that the vector of empirical
measures of the times that each walk makes a left and a right
transition forms a stochastic approximation process. This process is
related to the flow induced by a smooth vector field defined on the
product of two 1-simplices. It is therefore sufficient to consider the
planar dynamics defined by the restriction of the field to the unit
square. This together with the fact that the vector field has negative
divergence suffices to show that the limit set of the stochastic
approximation process corresponds to the set of equilibria of the
vector field. Section~\ref{sec:DS} presents a characterisation of the
equilibria in terms of the repulsion parameter $\beta$, and then shows
that the stochastic approximation process converges to stable
equilibria and does not converges to unstable
equilibria. Section~\ref{sec:recurrence} finally presents the proof of
Theorems~\ref{th:trans} and \ref{th:rec}. The proof of
Theorem~\ref{th:trans} is a straightforward application of the results
in Section~\ref{sec:DS}. By contrast, the proof of
Theorem~\ref{th:rec} is more involved. Beyond showing that the
proportion of times each walk makes a left and a right transition
converges toward $\frac{1}{2}$, the proof of Theorem~\ref{th:rec}
relies on an estimate for the
speed of convergence, a zero-one law and a coupling  argument.

\section{The dynamical system approach}\label{sec:DS}
\subsection{Stochastic approximations}
For $n \geq 0$, $i =1,2$, define
\begin{equation}\label{eqn:xis}
  \xi(n) = \big (\xi^1_\lstp(n), \xi^1_\rstp(n), 
  \xi^2_\lstp(n), \xi^2_\rstp(n) \big),
 \,\,\,
  \xi^i_{\lstp}(n) = \mathbf 1_{\{S_{n+1}^i - S_n^i = - 1\}},
  \,\,\,
  \xi^i_{\rstp}(n) = \mathbf 1_{\{S_{n+1}^i - S_n^i =  1\}},
\end{equation}
and then let
\begin{equation}\label{eqn:xxis}
 X^i_{\lstp}(n) = \frac{1}{n}\sum_{k=0}^{n-1} \xi^i_{\lstp}(k),
 \qquad
 X^i_{\rstp}(n) = \frac{1}{n}\sum_{k=0}^{n-1} \xi^i_{\rstp}(k),
\end{equation}
be the proportion of left and right transitions of the $i$-th walk up
to time $n$. Hereafter we denote by $X = \{X(n)\}_{n \geq 0}$ the
process determined by $X(n) = (X^1_\lstp(n), X^1_\rstp(n),
X^2_\lstp(n), X^2_\rstp(n))$, defined on a suitable probability space
$(\Omega, \sF, \Pb)$.

The process $X$ takes values on the set $\D =
\triangle\times\triangle$, which equals the two-fold Cartesian product
of the one-dimensional simplex $\triangle = \{x \in \R^2 \mid x_v \geq
0, \sum_v x_v = 1\}$.  We will hereafter use $(x^1_\lstp, x^1_\rstp,
x^2_\lstp, x^2_\rstp)$ to denote the coordinates of any point $x \in
\D$ and also $x^i = (x^i_\lstp, x^i_\rstp)$ for $i=1, 2$. Let $\TsD =
\{(x^1, x^2) \in \R^{2\times 2} \mid x^i_\lstp+x^i_\rstp = 0, i=1,2\}$
be the tangent space of $\D$. Now, let $\pi : \D \to\D$ be the map
\begin{equation}\label{eqn:pifirst}
  x \mapsto \pi(x)
  =
  \big(\pi^1_\lstp(x), \pi^1_\rstp(x),
  \pi^2_\lstp(x), \pi^2_\rstp(x)\big)
\end{equation}
where for $i=1,2$ and $v = \lstp, \rstp$,
\begin{equation}
 \label{eqn:pipsi}
  \pi^i_v(x) = \psi(2 x_v^j -1), \qquad j = 3 - i. 
\end{equation}
For further computations, it is also worth observing that, since
$x^j \in \triangle$,
\begin{equation}\label{eqn:pipsi-novo}
 \pi^i_v(x) =  \frac{e^{-\beta x^j_v}}{e^{-\beta x^j_\lstp} 
 	+ e^{-\beta x^j_\rstp}}.
\end{equation}

\begin{lemma}\label{lem:X_is_SA}
The process $X=\{X(n)\}_{n\geq 0}$ satisfies the following recursion
\begin{equation}
  \label{eqn:SA} 
  X(n+1) - X(n) = \gamma_n(F(X(n))+U_n)
\end{equation}
where
\begin{equation}
	\label{eqn:gamma_and_U}
  \gamma_n = \frac{1}{n+1},
  \qquad\qquad
  U_n = \xi(n) - \E[\xi(n)\mid\sF_n]
\end{equation}
and $F:\D \to \TsD$ is the vector field $F= (F^1_\lstp, F^1_\rstp,
F^2_\lstp, F^2_\rstp)$ defined by
\begin{equation}
 \label{eqn:TheField}
  F(X(n)) = -X(n) + \pi(X(n)).
\end{equation}
\end{lemma}

The proof of Lemma~\ref{lem:X_is_SA} is presented in the Appendix. A
discrete time process whose increments are recursively computed
according to (\ref{eqn:SA}) is known as a stochastic
approximation. Provided the random term $U_n$ can be damped by
$\gamma_n$, (\ref{eqn:SA}) may be thought as a Cauchy-Euler
approximation scheme, $x(n+1) - x(n) = \gamma_n F(x(n))$, for the
numerical solution of the autonomous ODE
\[
  \dot x = F(x).
\]
Under this perspective, a natural approach to determine the limit
behaviour of the process $X$ consists in studying the asymptotic
properties of the related ODE. This heuristic, known as the ODE
method, has been rather effective while studying various reinforced
stochastic processes.

Let $x = (x^1_\lstp, x^1_\rstp, x^2_\lstp, x^2_\rstp)$ be a generic
point of $\D$. By (\ref{eqn:TheField}), the ODE determined
by the stochastic approximation in our case is given by the equation 
\begin{equation}
  \label{eqn:ODE}
  \dot x = F(x) = -x + \pi(x).
\end{equation}
By using (\ref{eqn:pipsi-novo}), equation (\ref{eqn:ODE}) explicitly
reads as
\begin{equation}
  \label{eqn:ODEexplicit}
  \begin{aligned}
   &\frac{d}{dt} x^1_v
   = 
   -x^1_v + \frac{e^{-\beta x^2_v}}{e^{-\beta x^2_\lstp} 
     + e^{-\beta x^2_\rstp}},
   \\ %\qquad                                   
   &\frac{d}{dt} x^2_v
   = 
    -x^2_v + \frac{e^{-\beta x^1_v}}{e^{-\beta x^1_\lstp}
       + e^{-\beta x^1_\rstp}},
  \end{aligned}
  \qquad
  v = \lstp, \rstp.
\end{equation}
Because each $x^1$ and $x^2$ assume values on the one-dimensional
simplex $\triangle$, the system of four equations described by
\eqref{eqn:ODEexplicit} can be reduced to two equations; for instance,
those governing the evolution of $x^1_\lstp$ and
$x^2_\lstp$. The dynamics of the ODE in \eqref{eqn:ODE} can therefore
described on the unit square $[0, 1]^2$ by identifying the field
with its projection $F \equiv (F^1_\lstp, F^2_\lstp)$. This
observation allows to use the dynamical system approach to planar
stochastic approximations described in \cite{BH99} and
\cite{B99}. Theorem~\ref{th:X_as_convergence}, stated bellow, is a
consequence of this. It provides a crucial characterisation for the
asymptotic behaviour of the process $X$.

A point $x \in \D$ is an equilibrium of $F$ if $F(x) =
0$. The set of equilibria of $F$ will hereafter be denoted by
$\Eq$.

\begin{theorem}\label{th:X_as_convergence} Let $X = \{X(n)\}_{n\geq
0}$ be a process satisfying the recursion \eqref{eqn:SA}. For any
$\beta \in [0, \infty){\setminus}\{2\}$, the process $X$ converges
almost surely toward an equilibrium of the vector field $F$ defined in
\eqref{eqn:TheField}.
\end{theorem}

The proof of Theorem~\ref{th:X_as_convergence} is presented in
Section~\ref{sec:ProofTh3}. We present first a description of the
equilibria of the vector field $F$.

\subsection{Equilibria}
This section identifies the equilibria of vector field defined by
(\ref{eqn:TheField}) and further studies their stability depending on
the repulsion parameter $\beta$. For any point $x \in \D$, let
$\JF{x}$ be the Jacobian matrix of the vector field $F$ at $x$ and let
$\sigma(\JF{x})$ be the set of its eigenvalues.  The equilibrium $x$
is hyperbolic if all the eigenvalues of $\sigma(\JF{x})$ have non-zero
real parts. The hyperbolic equilibrium $x$ is linearly stable if
$\sigma(\JF{x})$ contains only eigenvalues with negative real parts;
otherwise $x$ is said to be linearly unstable.

\begin{lemma}\label{lem:EquilibriaForm}
Let $g : [0,1]\to [0,1]$ be a strictly decreasing function and such
that $g(1-w)=1-g(w)$ for all $w \in [0,1]$.  Let $E_1$ and $E_2$ be
the sets defined as
\begin{align*}
  &E_1 = \Big\{x\in\D\ \Big|\ x_v^i = g(x_v^j)\  \text{for all $i$,
    $v$ and  $j=3-i$} \Big\},  \\
  &E_2 = \Big\{x\in\D\ \Big|\ x = (w, 1-w, 1-w, w) \text{ where }
    w=g(1-w) \Big\}.
\end{align*} 	
Then $E_1 = E_2$.  In particular, $\big(\frac{1}{2}, \frac{1}{2},
\frac{1}{2}, \frac{1}{2}\big) \in E_2$, and $(w, 1-w, 1-w, w) \in E_2$
if and only if $(1-w, w, w, 1-w) \in E_2$.
\end{lemma}
\begin{proof}
Assume $x \in E_1$. First we show that $x_\lstp^1 =
x_\rstp^2$. Suppose by contradiction, and without loss of generality,
that $x_\rstp^2 < x_\lstp^1$. Since $g$ is strictly decreasing, we
would have that $1 - x_\rstp^2 = x_\lstp^2 = g(x_\lstp^1) <
g(x_\rstp^2) = x_\rstp^1 = 1- x_\lstp^1$, contradicting the hypothesis
that $x_\rstp^2 < x_\lstp^1$. Since $x_\lstp^1 = x_\rstp^2$, by
setting $w = x_\lstp^1 = x_\rstp^2$, we have that $x_\rstp^1 =
x_\lstp^2 = (1 - w)$. To conclude that $x \in E_2$, it sufficient to
observe that $w = g(1 -w)$. Indeed,
\[
  w=x^1_\lstp = g(x^2_\lstp) = g(1 - x^2_\rstp) =  g(1 -w). 
\]
The second inequality holds because $x \in E_1$ and the third, because
$x^2_\lstp = 1 - x^2_\rstp$.

Conversely, assume that $x \in E_2$. Then $(x^1_\lstp, x^1_\rstp,
x^2_\lstp, x^2_\rstp) = (w, 1-w, 1-w, w)$ for some $w$ with $w = g(1-
w)$. As an immediate consequence, we have that $x_\lstp^1 =
g(x_\lstp^2)$ and $x_\rstp^2 = g(x_\rstp^1)$. To conclude, we show
next that $x_\rstp^1 = g(x_\rstp^2)$ and $x_\lstp^2 =
g(x_\lstp^1)$. Indeed,
\[
  x_\rstp^1 = x_\lstp^2 = 1 - w = 1 - g(1 - w) = 1 - (1 - g(w)) = g(w) 
  = g(x_\lstp^1) = g(x_\rstp^2).
\]
The third equality holds because $x \in E_2$ and hence $w = g(1 -
w)$. The fourth equality holds by hypothesis on $g$, that is, $g(1 -
w) = 1 - g(w)$ for all $w$. The last two equalities follow because $w
= x_l^1 = x_r^2$.
\end{proof}

\begin{lemma}\label{lem:eqtypes} For $\beta \in [0,2]$, the point
$\big(\frac{1}{2}$, $\frac{1}{2}$, $\frac{1}{2}$, $\frac{1}{2}\big)$
is the only equilibrium for the vector field $F$ given by
\eqref{eqn:TheField}. For any $\beta > 2$, the field has three
equilibria,
\begin{equation}\label{eqtypes}
  \big(\textstyle\frac{1}{2}, \frac{1}{2}, \frac{1}{2},
  \frac{1}{2}\big),
  \quad
  (w, 1-w, 1-w, w)
  \quad \mbox{and}\quad
  (1-w, w, w, 1-w),
\end{equation}
where $w \in (0,  \frac{1}{2})$ is uniquely determined by $\beta$. The
equilibrium $\big(\frac{1}{2}$, $\frac{1}{2}$, $\frac{1}{2}$,
$\frac{1}{2}\big)$ is linearly stable for $\beta \in [0, 2)$ and linearly
unstable for $\beta > 2$. The equilibria $(w, 1-w, 1-w, w)$ and $(1-w,
w, w, 1-w)$ are linearly stable for $\beta > 2$.
\end{lemma}
\begin{proof}
Let $\Eq$ be the set of equilibria of the vector field given by
(\ref{eqn:TheField}), and $\psi$ be given as in
(\ref{eqn:psi}).  
First we show that $\Eq = E_2$, where $E_2$ is defined
as in Lemma \ref{lem:EquilibriaForm} for $g(w) = \psi(2 w -1)$. To
that end, note that $x \in \Eq$ if and only if $x_v^i = \pi_v^i(x) =
\psi(2 x_v^j -1) = g(x_v^j)$ for all $i$ and $v$, where $j = 3
-i$. This shows that $\Eq = E_2$.  Next, we show that $E_1 = E_2$. The
previous equality is ensured by Lemma \ref{lem:EquilibriaForm},
provided that $g$ is strictly decreasing and $g(1-w)=1-g(w)$ for all
$w \in [0,1]$. These two assertions follow immediately by inspection
on $g(w)$, where
\[
   g(w) =  \frac{1}{1+e^{2 \beta w - \beta}}.
\]
This shows that $\Eq = E_2$ with $g(w) = \psi(2w -1)$. In particular,
for all  $\beta \geq 0$, 
\begin{equation}
 \label{eqn:Def_of_E}
 \Eq =
 \Big\{x\in\D\ \Big|\ x = (w, 1-w, 1-w, w), \text{ where } w
 = g(1 - w)\Big\}.
\end{equation}

By Lemma \ref{lem:EquilibriaForm}, it follows that $(\frac{1}{2}$,
$\frac{1}{2}$, $\frac{1}{2}$, $\frac{1}{2}$) $\in \Eq$ and $(w, 1-w,
1-w, w) \in \Eq$ if and only if $(1-w, w, w, 1-w) \in \Eq$. To conclude,
it is sufficient to show two things. First, if $\beta \in [0,2]$, then
there is no $w \in [0,\frac{1}{2})$ such that $w = g(1-w)$; and
second, if $\beta \in (2,\infty)$, then there is only one $w \in
[0,\frac{1}{2})$ such that $w = g(1-w)$.

If $\beta = 0$, then the first assertion holds because $g(1 - w) =
\frac{1}{2}$ for all $w \in [0,\frac{1}{2})$. If $\beta > 0$, then
both assertions hold because $g(1 - w)$ is bounded from below by zero,
increasing, strictly convex on $[0,\frac{1}{2})$, and such that
\[
  \frac{\partial}{\partial w} g(1 - w)\big{|}_{w=\frac{1}{2}} > 1
  \qquad \text{if and only if $\ \ \beta > 2$}.
\]

The stability of an isolated equilibrium point is determined by
studying a linearization of the vector field provided by the Jacobian
matrix at that point. For any $x \in \D$ let $\JF{x} = [\partial
F^i_k(x)/\partial x^j_s]$ for $i=1,2$, $j=1,2$, and $k, s \in
\{\lstp,\rstp\}$.  For $x_* =(\frac{1}{2}, \frac{1}{2}, \frac{1}{2},
\frac{1}{2}) \in \Eq$, the Jacobian matrix is
\[
 \JF{x_*}
 =
 \begin{bmatrix*}[r]
  -1      &    0    & -\frac{\beta}{4}  & \frac{\beta}{4}   \\[.4em]
   0      &   -1    &  \frac{\beta}{4}  & -\frac{\beta}{4}  \\[.4em]
  -\frac{\beta}{4}  &  \frac{\beta}{4}  & -1       &    0   \\[.4em]
   \frac{\beta}{4}  & -\frac{\beta}{4}  &  0       &   -1
 \end{bmatrix*}.
\]
The four eigenvalues of $\JF{x_*}$ are easily computed and equal
\begin{equation}\label{eqn:eingenv}
-1, \quad  -1, \quad  -1 - \frac{\beta}{2}, \quad\text{and }\  -1 +
\frac{\beta}{2}.
\end{equation}
This shows that the equilibrium $x_* =(\frac{1}{2}, \frac{1}{2},
\frac{1}{2}, \frac{1}{2})$ is linearly stable if $\beta < 2$ and
linearly unstable if $\beta > 2$.

Now, suppose that $\beta > 2$, and let $x_w = (w, 1-w, 1-w, w) \in
\Eq$, where $w \in (0,\frac{1}{2})$. The Jacobian of the vector field
at $x_w$ is given in this case by the matrix
\[
 \JF{x_w}
 =
 \begin{bmatrix*}[r]
   -1     &    0    & -h(w,\beta)  &  h(w,\beta)   \\[.4em]
   0      &   -1    &  h(w,\beta)  & -h(w,\beta)   \\[.4em]
  -h(w,\beta)   &  h(w,\beta)   & -1     &    0    \\[.4em]
   h(w,\beta)   & -h(w,\beta)   &  0     &   -1
 \end{bmatrix*},
\]
where
\[
 h(w, \beta) = \frac{\beta}{2+2\cosh(\beta-2 w \beta)}.
\]
Two eigenvalues of this matrix equal $-1$. The two other eigenvalues
are $-1 \mp 2 h(w,\beta)$. Simple analysis shows that the eigenvalue
$-1-2h(w,\beta)$ is negative for any $\beta > 2$ and $w \in
\big[0,\frac{1}{2}\big)$.  To conclude that $x_w$ is stable, it remains
to show that $-1 + 2h(w,\beta)$ is negative.  Note that, for $\beta >
2$ and $v \in \big[0, \frac{1}{2}\big)$, the map $v\mapsto
-1+2h(v,\beta)$ is increasing and equals 0 at a single value $w_*$
determined by
\[
  w_* = \frac{\beta - \text{arcosh}(\beta-1)}{2\beta}.
\]

To conclude that $-1+2h(w,\beta)$ is negative, we show that $w < 
w_*$. A straightforward computation shows that $w_*$ is the unique
solution to
\[
  \frac{\partial}{\partial w_*} g(1-w_*) = 1 \quad\text{for}\  w_* \in
  \Big[0,\frac{1}{2}\Big).
\]
Since $x_w = (w, 1-w, 1-w, w) \in \Eq$ and $w \in
\big[0,\frac{1}{2}\big)$, we have that $g(1-w) = w$, where $g$ is the
map used in the definition of $\Eq$ in (\ref{eqn:Def_of_E}). Since
$g\big(\frac{1}{2}\big) = \frac{1}{2}$, $w \mapsto g(1- w)$ is
continuous, strictly increasing and strictly convex for $w \in \big[0,
\frac{1}{2}\big]$, it follows that $w < w_*$.

An analogous argument shows that the equilibrium $(1-w, w, w, 1-w)$ is
stable when $\beta > 2$, because the Jacobian of the vector field at
this point has the same spectrum as the Jacobian at $(w, 1-w, 1-w,
w)$.
\end{proof}

\subsection{Convergence to equilibria} 
\label{sec:ProofTh3}
This section presents the proof of Theorem~\ref{th:X_as_convergence}.
Its proof relies on Lemma~\ref{lem:LimitSet-second-part} stated
bellow. This lemma allows us to relate the limiting behaviour of the
random process $X$ to the one of the flow induced by the vector field
in \eqref{eqn:TheField}. We will make use of the following
terminology, mostly taken form \cite{B99}, to state this result.

A semi-flow on $\D$ is a continuous map $\phi:\R_+\times\D\to\D$ such
that $\phi_0$ is the identity on $\D$, and $\phi_{t+s} =
\phi_t\circ\phi_s$ for any $t, s \geq 0$. To simplify notation we used
$\phi_t(x)$ instead of $\phi(t, x)$. A subset $A\subset \D$ is said to
be positively invariant if $\phi_t(A)\subset A$ for all $t\geq 0$. Let
$F$ be a continuous Lipschitz vector field on $\D$. The semi-flow
induced by $F$ is the unique smooth map $\Phi=\{\phi_t\}$ such that:
1. $\phi_0(x_0) = x_0$ for any $x_0 \in \D$, and 2.
$\frac{d}{dt}\phi_t(x_0) = F(\phi_t(x_0))$ for all $t \geq 0$.

A simple verification shows that vector field $F$ in
(\ref{eqn:TheField}) is Lipschitz continuous, hence the induced
semi-flow $\Phi$ is uniquely determined by $F$. Moreover, the
following lemma shows that $\D$ is positively invariant by $\Phi$.

\begin{lemma}\label{lem:DomainInvariance}
$\D$ is positively invariant for the semi-flow $\Phi$ induced by the
vector field $F$ in (\ref{eqn:TheField}).
\end{lemma}

\begin{proof} Let $z = (z_\lstp^1, z_\rstp^1, z_\lstp^2, z_\rstp^2)$
be a generic point in $\D$. Suppose that $z \in \partial\D$ and hence,
without loss of generality, that $z_\lstp^1 = 0$ and $z_\rstp^1 =
1$. Suppose $\phi_t$ is a solution of (\ref{eqn:ODE}) with $\phi_0 =
z$. For any $t \geq 0$, write $\phi(t, z) = \phi_t(z)$. Since $F(z)
\in \TsD$, it is sufficient to show that $\frac{d }{dt}\phi_\lstp^1(t,
z)\big|_{t=0} > 0$, in which case, it holds also that $\frac{d }{dt}
\phi_\rstp^1(t, z)\big|_{t=0} = - \frac{d }{dt} \phi_\lstp^1(t, z)\big|_{t=0}
< 0$. By (\ref{eqn:pifirst}), (\ref{eqn:pipsi-novo}), and
(\ref{eqn:ODE}), it follows that
\[
  \frac{d }{dt} \phi_\lstp^1(t, z)|_{t=0} =  \psi(2 z_\lstp^2 - 1) \geq
  \inf_y \psi(2 y - 1) 
  =
  \frac{1}{1+e^{\beta}} > 0.
\]
This shows that $F(z)$ points inwards whenever $z \in \partial\D$,
and hence that $\phi_t \in \D$ for all $t>0$ if $\phi_0 \in \D$.
\end{proof}

In terms of the semi-flow $\Phi$, a point $x\in \D$ is said to be an
equilibrium if $\phi_t(x) = x$ for all $t \geq 0$. A point $x \in \D$
is periodic if $\phi_T(x) = x$ for some $T>0$. The set $\gamma(x) =
\{\phi_t(x)\, :\, t \geq 0\}$ is the orbit of $x$ by $\Phi$. A subset
$\Gamma \subset \D$ is a orbit chain for $\Phi$ provided that for some
natural number $k\geq 2$, $\Gamma$ can be expressed as the union
$\Gamma = \{e_1, \ldots, e_k\}\bigcup \gamma_1 \bigcup \ldots \bigcup
\gamma_{k-1}$ of equilibria $\{e_1, \ldots, e_k\}$ and nonsingular
orbits $\gamma_1$, $\ldots$, $\gamma_{k-1}$ connecting them. If $e_1 =
e_k$, $\Gamma$ is called a cyclic orbit chain.

Let $\delta >0$, $T>0$. A $(\delta, T)$-pseudo orbit from $x \in \D$
to $y \in \D$ is a finite sequence of partial orbits $\{\phi_t(y_i) :
0 \leq t \leq t_i\}$; $i=0, \ldots, k-1$; $t_i\geq T$ of the semi-flow
$\Phi= \{\phi_t\}_{t\geq 0}$ such that
\[
  \|y_0 - x\| < \delta,
  \qquad
  \|\phi_{t_i}(y_i) -  y_{i+1}\| < \delta,\ \ i=0, \ldots, k-1,
  \quad\text{and}\quad
  y_k = y.
\]
A point $x \in \D$ is chain-recurrent if for every $\delta>0$
and $T>0$ there is a $(\delta, T)$-pseudo orbit from $x$ to
itself. The set of chain-recurrent points of $\Phi$ is denoted by
$\CR$. The set $\CR$ is closed, positively invariant by $\Phi$ and
such that $\Eq \subset \CR$.

Let $\sL\big(\{X(n)\}\big)$ be the limit set of the stochastic
approximation process $X = \{ X(n)\}_{n\geq 0}$. That is, for any point
$\omega \in \Omega$, the value of $\sL\big(\{X(n)\}\big)$ at $\omega$
is given by the set of points $x \in \R^{md}$ for which $\lim_{k\to
\infty} X(n_k, \omega) = x$, for some strictly increasing sequence of
integers $\{n_k\}_{k \in \N}$. 

Next we show that $\sL\big(\{X(n)\}\big)$ is almost surely connected
and included in $\CR$. This is the content of
Lemma~\ref{lem:LimitSet-second-part}. To show
Lemma~\ref{lem:LimitSet-second-part}, we use the following lemma. 

\begin{lemma}\label{lem:Kushner} 
Let $X = \{X(n)\}_{n\ge 0}$ be a process satisfying the recursion in
\eqref{eqn:SA} such that $F$ defined by \eqref{eqn:TheField} is a
continuous vector field with unique integral curves. Then
\begin{enumerate}[(i), nosep]
\item $\{X(n)\}_{n\geq 0}$ is bounded, 
\item $\lim_{n\to\infty}\gamma_n=0$, $\sum_{n\geq 0}
    \gamma_n = \infty$, 
\item\label{as:KushnerLemma} for each $T>0$, almost surely it holds
that 
\[
  \lim_{n\to\infty}
  \Bigg(\sup_{\{\, r\, :\, 0\, \leq\, \tau_r - \tau_n\, \leq\, T\, \}}
  \Bigg\|\sum_{k=n}^{r-1} \gamma_k U_k\Bigg\|
  \Bigg) = 0,
\]
where $\tau_0=0$ and $\tau_n = \sum_{k=0}^{n-1} \gamma_k$.
\end{enumerate}
\end{lemma}

\begin{proof}
Item (\emph{i}) follows by definition of $\{X(n)\}_{n\geq 0}$ in
\eqref{eqn:xxis}.  Item (\emph{ii}) is immediate by the form of
$\gamma_n$ in (\ref{eqn:gamma_and_U}). The proof of (\emph{iii}) is
presented in the Appendix.
\end{proof}

\begin{lemma}\label{lem:LimitSet-second-part} 
Let $X = \{X(n)\}_{n\ge 0}$ be a process satisfying the recursion in
\eqref{eqn:SA} and $\CR$, the chain-recurrent set of the semi-flow
induced by the vector field $F$ in \eqref{eqn:TheField}. Then,
$\mathfrak L\big(\{X(n)\}\big)$ is almost surely connected and
included in $\CR$.
\end{lemma}	

\begin{proof} Since $X$ satisfies the properties
(\emph{i})-(\emph{iii}) in Lemma~\ref{lem:Kushner}, the proof of the
lemma follows from Theorem 1.2 in \cite{B96}.
\end{proof}

We are now in the position to present the proof of
Theorem~\ref{th:X_as_convergence}.

\begin{proof}[Proof of Theorem~\ref{th:X_as_convergence}]
We show first that $\CR \subset \Eq$. Let $\Phi=\{\phi_t\}_{t\geq 0}$
denote the planar semi-flow induced by the vector field $F\equiv
(F^1_\lstp, F^2_\lstp)$, where $F^1_\lstp$ and $F^2_\lstp$ are two of
the coordinate functions of the field defined by
\eqref{eqn:TheField}. By Lemma~\ref{lem:eqtypes}, we have that the
field $F$ has isolated equilibria. It then follows from Theorem 6.12
in \cite{B99} that for any point $p \in \CR$ one of the following
holds:
\begin{enumerate}[($a$), nosep]
\item $p$ is an equilibrium
\item $p$ is periodic
\item There exists a cyclic orbit chain $\Gamma \subset \CR$ which
contains $p$.
\end{enumerate} A simple computation shows that div$F$, the divergence
of $F$, is negative, indeed
$
   \text{div} F(x) = \partial F_\lstp^1(x)/\partial x_\lstp^1
   +  \partial F_\lstp^2(x)/\partial x_\lstp^2 = - 2.
$
This implies that $\phi_t$ decreases area for $t > 0$. In this case,
according to Theorem 6.15 in \cite{B99}, it follows that:
\begin{enumerate}[1., nosep]
\item $\CR$ is a connected set of equilibria which
  is nowhere dense and which does not separate the plane
\item If $\Phi$ has at most countably many equilibrium points, then 
  $\CR$ consists of a single stationary point.
\end{enumerate}
Both options, ($b$) and ($c$), are therefore ruled out and hence $\CR 
\subset \Eq$.

Observe now that, by Lemma~\ref{lem:LimitSet-second-part}, we have
that $\mathfrak L\big(\{X(n)\}\big)$ is almost surely connected and
included in $\CR$. Since $\CR \subset \Eq$ and $\Eq$ is formed by
isolated points it follows that $X(n)$ converges almost surely towards
a point of $\Eq$.
\end{proof}

\subsection{Non-convergence to the unstable equilibrium}
\label{sec:NonConvrg}
A step to characterise the asymptotic behaviour of the stochastic
approximation $X$ consists in establishing that this process does not
converges toward linearly unstable equilibria of $F$. This is
accomplished here by using Theorem 1 in~\cite{P90}, evoqued in the
proof of the following lemma.

\begin{lemma}\label{non-covergence-int} Let $X=\{X(n)\}_{n\ge 0}$ be a
process satisfying the recursion in (\ref{eqn:SA}). Then, if $\beta
> 2$,
\[
  \Pb\Big(\lim_{n\to\infty} X(n)
  =
  \big(\textstyle\frac{1}{2}, \frac{1}{2}, \frac{1}{2},
  \frac{1}{2}\big)\Big)
  = 0.
\]	
\end{lemma}
\begin{proof} Let $x_*=\big(\frac12, \frac12, \frac12, \frac12\big)$
and $U_n$ be defined as in Lemma~\ref{lem:X_is_SA}. Throughout,
$\|\cdot\|$ stands for the $L^1$ norm in $\R^4$. The proof follows
from Theorem 1 in \cite{P90}, provided that the following conditions
are satisfied:
\begin{enumerate}[(i), nosep]
\item $x_*$ is a linearly unstable critical point of $F$,
\item $\|U_n\| \leq c_1$ for some posite constant $c_1$, and
\item For every $x \in \mathcal B(x_*)$, $n > n_0$, and $\theta
\in \TsD$ with $\|\theta\|=1$, there is a postitive constant $c_2$
such that 
\[
  \E\Big[\textup{max}\big\{\big\langle\theta, U_n\big\rangle,0\big\}
  \,\Big|\, X(n) = x, \sF_n\Big] \geq c_2.
\]
\end{enumerate}

Condition (\emph{i}) follows immediately from Lemma~\ref{lem:eqtypes}
and Condition (\emph{ii}), by the definition of $U_n$ in
(\ref{eqn:gamma_and_U}). The rest of the proof concerns the
verification of (\emph{iii}).

Let $\TsD_1 = \{ \theta \in \TsD \, : \, \|\theta\| = 1 \}$ and $n_0$
be defined as in the first paragraph of the introduction. For $w \in
\R$, let $w^+ = \textup{max}\{w,0\}$. It is sufficient to show that,
for all $n > n_0$, $x \in \D$, and $\theta \in \TsD_1$, we have that
\begin{equation}\label{Edesig}
   \E \Big[ \big\langle\theta, U_n 
   \big\rangle^+ \Big{|} \,X(n) = x, \, \sF_n \Big] \geq s(x) 
\end{equation}
where $s:\D \to \mathbb{R}$ is a continuous function with $s(x_*) >
0$.

Let 
\begin{equation}
s(x) = \frac{1}{2} \Big(\min_{i,v} \pi_{v}^{i} \big(x\big)
\Big)^{3}.
\end{equation}
Clearly $s$ is continuous because $\pi_v^i$ are continuous. Since
$F(x) = -x + \pi(x)$ and since $F(x_*) = 0$, we have that $\pi(x_*) =
x_* = \big(\frac12, \frac12, \frac12, \frac12\big)$ and, therefore,
$s(x_*) > 0$.

It remains to show (\ref{Edesig}).  Let $\theta \in \TsD_1$. For each
walk $i \in \{1,2\}$, choose a vertex $v^i \in \{\lstp, \rstp\}$, such
that
\[
  \theta_{v^i}^i  =  \max_v \theta_v^i.
\]
Next, define the event $A = \bigcap_{i = 1, 2} \{\xi_{v^i}^i(n) =
1\}$, with $\xi$ as defined by (\ref{eqn:xis}). That is, $A$ is the
event in which walk $i \in \{1,2\}$ makes a transition to vertex $v^i$
 at time $n+1$. For all
$n \geq n_0$ and $\theta \in \TsD_1$, we have that
\begin{equation}\label{Edesig33}
\E\Big[\big\langle\theta, U_n\big\rangle^+  \Big  | \, X(n) = x, \sF_n
\Big ] 
=
\E\Big[\big\langle\theta, U_n\big\rangle^+  \Big  | \, X(n) = x \Big
] 
\geq
q(x, \theta)
\end{equation}
where
\begin{equation}
 \label{eqn:def_of_q}
  q(x, \theta)
 =
  \E\Big[\big\langle\theta, U_n\big\rangle^+  \Big  | \, A,\, X(n) = x 
  \Big]\Pb\big(A\, | \, X(n) = x\big).    
\end{equation}
Note that the first equality
follows because the distribution of $U_n$ is uniquely
determined by $X(n)$ according to (\ref{eqn:gamma_and_U}).  The
inequality in (\ref{Edesig33}) holds because $\langle\theta,
U_n\rangle^+$ is non-negative. Now, to show (\ref{Edesig}), it is
sufficient to prove that for all $\theta \in \TsD_1$ and $x \in \D$
\begin{equation}\label{inetheta}
    q(x, \theta) \geq s(x).
\end{equation}
Assume without loss of generality, that $v^i = \lstp$, $i = 1,
2$. That is, $\theta \in \TsD_1$ is of the form $(\theta_{\lstp}^1,
\theta_{\rstp}^1, \theta_{\lstp}^2, \theta_{\rstp}^2) = \big(a, -a,
\frac12 - a, a -\frac12\big)$ for some $a \in \big[0,
\frac12\big]$. In that case, $A = \{\xi_{\lstp}^1(n) = 1, \,
\xi_{\rstp}^1(n) = 0, \,\xi_{\lstp}^2(n) = 1, \, \xi_{\rstp}^2(n) =
0\}$. According to (\ref{eqn:gamma_and_U}), we have $(U_n)_v^i =
\xi_v^i(n) - \E[\xi_v^i(n)\mid\sF_n] = \xi_v^i(n) -
\pi_v^i(X(n))$. Using the previous equality and the particular
form of $\theta$ and $A$, it follows by the definition of $q$ in
\eqref{eqn:def_of_q} that
\begin{align*}
 q(\theta, x)
 %&=
 %\E\Big[\big\langle\theta, U_n\big\rangle^+  \Big  | \, A,\, X(n) = x 
%\Big]\Pb\big(A\, | \, X(n) = x\big). \\
 %
 %
 %
 %\Big[\sum_{i,v} \theta_v^i \Big (\xi_v^i(n) - \pi_v^i(x)
 %\Big)\Big]^+  \Pb\big(A\, | \, X(n) = x \big)  \\
 &=
   \Big[\sum_{i} \theta_{\lstp}^i - \sum_{i,v} \theta_v^i
   \pi_v^i(x)\Big ]^+  \Pb\big(A\, | \, X(n) = x \big) \\
%  &=
%    \Big[\frac{1}{2} - \sum_{i,v} \theta_v^i  \pi_v^i(x)\Big]^+
%    \Pb\big(A\, | \, X(n) = x \big) \\
  &=
    \Big[\frac{1}{2} - \sum_{i,v} \theta_v^i  \pi_v^i(x)\Big]^+
    \prod_{i=1}^2\pi_{\lstp}^{i}(x) \\
  &\geq
    \Big[\frac{1}{2} - \sum_{i,v} \theta_v^i  \pi_v^i(x)\Big ]^+
    \Big(\min_{i,v}\pi_v^{i}(x) \Big)^2,
\end{align*}
where the last equality uses the fact that the transitions of the
walks are independent given the event $\{X(n) = x\}$ and, therefore,
$\Pb\big(A\, | \, X(n) = x \big) = \prod_{i=1}^2 \pi_{\lstp}^{i}
(x)$.

To show (\ref{inetheta}), it is sufficient to show that $(\frac{1}{2}
- \sum_{i,v} \theta_v^i \pi_v^i(x))^+ \geq \frac{1}{2} \min_{i,v}
\pi_v^{i}(x)$. To simplify notation, set $\pi_v^{i} =
\pi_v^{i}(x)$. Since $(\theta_{\lstp}^1, \theta_{\rstp}^1,
\theta_{\lstp}^2, \theta_{\rstp}^2) = \big(a, -a, \frac12 - a, a
-\frac12\big)$, it
follows that
\begin{align*}
  \frac{1}{2} - \sum_{i,v} \theta_v^i \pi_v^i
  &= \frac{1}{2} - \Big(a
    \pi_{\lstp}^1 + \Big(\frac{1}{2} - a\Big)\pi_{\lstp}^2\Big) + a
    \pi_{\rstp}^1 + \Big(\frac{1}{2} - a\Big)\pi_{\rstp}^2 \\
  &\geq \frac{1}{2} -
    \Big(a \pi_{\lstp}^1 + \Big(\frac{1}{2} - a\Big)\pi_{\lstp}^2\Big)
    + a \min_{i,v}\pi_v^i + \Big(\frac{1}{2} - a\Big)
    \min_{i,v}\pi_v^i \\
  &\geq
    \frac{1}{2} \min_{i,v}\pi_v^i,
\end{align*}
where the last inequality uses the fact that $a \in \big[0, \frac12\big]$,
$\pi_\lstp^1, \pi_\lstp^2, \in [0, 1]$, and, therefore, $\frac{1}{2} -
\big(a \pi_{\lstp}^1 + (\frac{1}{2} - a)\pi_{\lstp}^2 \big) \geq
\frac{1}{2} - \big(a + (\frac{1}{2} - a) \big) = 0$.
\end{proof}

\section{Proof of Theorems~\ref{th:trans} and
  \ref{th:rec}}\label{sec:recurrence}
This section presents the proof of the transience of both walks
$S^i_n$, $i=1, 2$, when $\beta \in (2,\infty)$, and the recurrence
when $\beta \in [0,1]$. The problem that arises when $\beta \in (1,2)$
is mentioned in Remark~\ref{remark:speed} at the end of this section.

The transience will make use of the following lemma.

\begin{lemma}\label{th:differenceConv} There is a unique point $x 
  \in [0,1]$, depending on $\beta$, such that,
\[
  \lim_{n\to\infty}\frac{1}{n} \Big(S_n^1-S_{n_0}^1,\, 
  S_n^2-S_{n_0}^2\Big)\ \in\  \Big\{ (x, -x),\, (-x, x)\Big\}
  \qquad
\textup{a.s.}
\]
In addition, if\ \  $0 \leq \beta \leq 2$, then $x = 0$, and if $\beta >
2$, then $0 < x < 1$.
\end{lemma}

\begin{proof}
From Theorem~\ref{th:X_as_convergence}, $X$ converges almost surely
towards one element of the set $\Eq$, the set of equilibria of the
vector field $F$ in \eqref{eqn:TheField}. This set is characterised by
Lemma~\ref{lem:eqtypes}.  Noting that
\begin{equation}
  \label{eqn:Sn_increment}
  (S^i_n - S^i_{n_0})/n = 2 X^i_\rstp(n) - 1,
\end{equation}
it follows by Lemma~\ref{lem:eqtypes} that $X^i_\rstp(n)
\longrightarrow \frac{1}{2}$ a.s. for $0 \leq \beta < 2$. As a
consequence, when $0 \leq \beta < 2$ we have that
\[
  \frac{S_n^i - S_{n_0}^i}{n} \longrightarrow 0 \,\, \,\, \text{a.s.}
\]

For the case $\beta > 2$, by Lemma~\ref{lem:eqtypes} and
Lemma~\ref{non-covergence-int} there is $w \in [0, \frac12)$ such that
\[
  \lim_{n\to\infty}
  \big(X^1_\rstp(n), X^2_\rstp(n)\big)
  \ \ \in\ \
  \big\{(w, 1-w), (1-w, w)\big\}\qquad\textup{a.s.}
\]
Using \eqref{eqn:Sn_increment} and setting $x=2w-1$, it follows
that $\lim_{n\to\infty} (S_n^i-S_{n_0}^i)/n \in \big\{(x,-x),
(-x,x)\big\}$ with $0 < x < 1$. This concludes the proof of the lemma.
\end{proof}

\begin{proof}[Proof of Theorem~\ref{th:trans}]
The proof of the theorem is an immediate consequence of
Lemma~\ref{th:differenceConv}.  Indeed, if $\beta > 2$, by
Lemma~\ref{th:differenceConv} it follows that $\big(S_n^1/n,
S_n^2/n\big)$ converges a.s. to $(x, -x)$ or $(-x, x)$ where $x >
0$. Therefore we have that either $(S_n^1, S_n^2) \to (+\infty,
-\infty)$ a.s. or $(S_n^1, S_n^2) \to (-\infty, +\infty)$ a.s.
\end{proof}

The rest of this section is devoted to the proof of
Theorem~\ref{th:rec}, that is, of the recurrence of $S^1_n$ and
$S^2_n$ when $\beta \in [0, 1]$.  Observe that in this case, according
to Lemma~\ref{lem:eqtypes}, the only equilibrium of $F$ is the point
$x_* = \big(\frac{1}{2}, \frac{1}{2}, \frac{1}{2},
\frac{1}{2}\big)$. We argue that both walks $S_n^i$ are recurrent
provided the process  $X$ converges
sufficiently fast towards $x_*$. To this end, we will make use of
several lemmas. The first of these  provides a rate of
convergence of $X$ towards $x_*$ when $\beta \in [0, 1]$. This is
obtained by considering the rate at which $\phi_t(x)$ converges
towards $x_*$ and the rate for the almost sure convergence of $X$
toward the trajectories of $\Phi$. The latter relies on the shadowing
techniques described in Section 8 of \cite{B99}. The proof follows
along the lines of the proof of Lemma 3.13 in \cite{BRS13}.

\begin{lemma}\label{lem:speedconv}
If $\beta \in [0, 1]$, then
\[
   \big\|X(n) - x_*\big\| = \mathcal O\Big(\frac{1}{\sqrt{n}}\Big)
 \qquad 
 \textup{a.s.}
\]
\end{lemma}

\begin{proof} Lemma~\ref{lem:eqtypes} shows that $x_*$ is the only
equilibrium of $F$ when $\beta \in [0,1]$.  This lemma also shows that
$x_*$ is hyperbolic and linearly stable. By
Theorem~\ref{th:X_as_convergence} it then follows that a.s. $X(n) \to
x_*$. Further, according to Theorem 5.1 in \cite{R99}, p. 153, we have
that the equilibrium $x_*$ is exponentially attracting. More
precisely, there is a neighbourhood $\mathcal{U} \subset \D$ of $x_*$
and two constants $C\geq 1$, $\zeta > 0$, such that for any initial
condition $x \in \mathcal{U}$, the solution $\phi_t$ of
(\ref{eqn:ODE}) satisfies
\begin{equation}
  \label{eqn:exp_rate}
  \big\|\phi_t(x) - x_*\big\| \leq Ce^{-t  \zeta}
  \big\|x - x_*\big\| \quad
  \text{ for all }\ t \geq 0.
\end{equation}
The constant $\zeta$ is such that $-\zeta$ is an upper bound for all
the eigenvalues $\lambda$ of $\JF{x_*}$, that is,
$\Re\text{e}(\lambda) \leq -\zeta < 0$. We observe that because of
(\ref{eqn:eingenv}), here we have the explicit expression $\zeta = 1 -
\beta/2$.

Let $\tau_n = \sum_{k=1}^n \gamma_k$ and let $Y:\R^+ \to \D$ be a
continuous time piecewise affine process defined such that: (\emph{i})
$Y(\tau_n) = X(n)$ and (\emph{ii}) $Y$ is affine on $[\tau_n,
\tau_{n+1}]$. $\{Y(t)\}_{t\geq 0}$ may be defined as the following
linear interpolation of $X(n)$,
\[
  Y(\tau_n + s) = X(n) + s\frac{X(n_1)-X(n)}{
    \tau_{n+1} - \tau_{n}},
  \quad\text{for}\quad
  0 \leq s \leq \gamma_{n+1}, \quad n \geq 0.
\]
By Proposition 8.3 in \cite{B99}, the interpolated process
$\{Y(t)\}_{t \geq 0}$ is almost surely a
$-\frac{1}{2}$-pseudotrajectory of $\Phi$, that is,
\begin{equation}
  \label{eqn:shadow_rate}
  % \sup_{T > 0}
  \limsup_{t\to\infty} \frac1t \log\bigg(\sup_{0\leq h
    \leq T}
  \big\| \phi_h\big(Y(t)\big)-Y(t+h)\big\|\bigg) \leq -\frac{1}{2}
\end{equation}
for all $T > 0$.

In view of (\ref{eqn:exp_rate}) and (\ref{eqn:shadow_rate}), by Lemma
8.7 in \cite{B99}, it follows that
\[
  \limsup_{t\to\infty} \frac1t \log\big(\big\| Y(t) - x_*\big\|\big)
  \leq -\min\Big\{\frac12, \zeta\Big\}.
\]
This in turn implies that
\[
  \|X(n) - x_*\| = \mathcal{O}\bigg(n^{-\min\big\{\frac12,\, 
    \zeta\big\}}\bigg)
\]
and hence concludes the proof because $\zeta \in
\big[\frac12, 1\big]$ when $\beta \in [0, 1]$.
\end{proof}

\begin{lemma}\label{lem:Pn_rate}
If $\beta \in [0, 1]$, then, for any $i =1,2$, $v=\rstp, \lstp$,
\[
   \Big|\pi_v^i \big (X(n) \big ) - \frac{1}{2}\Big| = 
   \mathcal{O}\Big(\frac{1}{\sqrt{n}}\Big)
   \qquad 
 \textup{a.s.} 
\]
\end{lemma}

\begin{proof}
Let $x_* = \big(\frac{1}{2}, \frac{1}{2}, \frac{1}{2},
\frac{1}{2}\big)$. By the definition of $\pi$, namely by equations
(\ref{eqn:psi}) and (\ref{eqn:pipsi}), we have that $\| \nabla \pi_v^i
\big (x_* \big )\|_{\infty} = \beta/2 \leq \frac12$. By linearization
of $\pi_v^i \big (X(n) \big )$ at $x_*$ it follows that $\pi_v^i \big
(X(n) \big ) - \pi_v^i \big (x_* \big ) = \big\langle \nabla
\pi_v^i(x_*) ,X(n) - x_*\big\rangle + R(X(n))$, where $R(X(n))$ is the
error of the approximation. Therefore 
\begin{align*}
  \Big|\pi_v^i \big (X(n) \big ) - \frac{1}{2} \Big| 
 &= 
  \big|\pi_v^i \big (X(n) \big ) - \pi_v^i \big (x_* \big ) \big| \\ 
 &\leq 
   \big\|\nabla \pi_v^i(x_*)\big\|_{\infty} \big\|X(n) - x_*\big\|
   + \big\|R(X(n)) \big\|   \\
 &\leq
   \bigg(\frac12 + \frac{\big\|R(X(n)) \big\|}{\big\| X(n)
   - x_*\big\|}\bigg)\big\|X(n) - x_*\big\|
\end{align*}
The proof is concluded by applying Lemma~\ref{lem:speedconv}
and observing that $\|R(X(n))\|/\|X(n) - x_*\|$ converges to zero as
$X(n)$ approaches $x_*$.
\end{proof}

\begin{corollary}\label{cor:Aepsilon}
Let $P_n = \pi_r^i \big (X(n) \big )$ for some fixed $i \in
\{1,2\}$. For each $\varepsilon > 0$, there are sufficiently large $b$
and $m$, depending on $\varepsilon$, such that
\begin{equation*}
\Pb(A) > 1 - \varepsilon, \, \text{ where } \,  A =
\bigg \{ \Big | 
P_n - \frac{1}{2} \Big | \leq \frac{b}{\sqrt{n}} \,\,\, \text{for
  all}  \,\,\,  n > m \bigg\}.
\end{equation*}
\end{corollary}

\begin{proof}
Let $\varepsilon > 0$ be arbitray. According to Lemma
\ref{lem:Pn_rate}, there is a set $\Omega^1 \subset \Omega$ with
$\Pb(\Omega^1) = 1$, such that $ \big|P_n(\omega) - \frac{1}{2}\big| =
\mathcal{O}\big(\frac{1}{\sqrt{n}}\big)$ for each $\omega \in
\Omega^1$. Therefore, for each $\omega \in \Omega^1$, there are well
defined constants $n(\omega) > 0$ and $b(\omega) > 0$ such that
\[
   \Big|P_n(\omega) - \frac{1}{2}\Big| \leq
   \frac{b(\omega)}{\sqrt{n}} \,\, \text{ for all } \,\, n >
   n(\omega). 
\]

Define $\Omega_k = \big\{ \omega \in \Omega^1 \, \big | \, \max 
\{b(\omega), n(\omega) \} \leq k \big\}$, $k =1, 2, 3,\ldots$ Note
that $(\Omega_k)_{k =1}^\infty$ is an increasing sequence of sets that
converges to $\Omega^1$. Since $\Pb(\Omega^1) = 1$, by continuty of
the probability measure, there is a sufficiently large $k_\varepsilon
> 0$ such that $\Pb(\Omega_{k_\varepsilon}) > 1 - \varepsilon$. Since
$A \supset \Omega_{k_\varepsilon}$ provided that $b > k_\varepsilon$
and $m > k_\varepsilon$, it follows that $\Pb(A) > 1 - \varepsilon$
for $b > k_\varepsilon$ and $m > k_\varepsilon$.
\end{proof}

\begin{lemma}\label{lemma:recurrence_pnn} Let $b > 0$ and $m > 0$,
and define $\{Z_n\}_{n \geq 0}$ as a non homogeneous random walk with
independent increments, parametrized by $b$ and $m$, as follows. Set
$Z_{n} = Z_0 + \sum_{k=1}^{n} Y_k$, where $Z_0 \in \mathbb Z$ and
$Y_1, Y_2, \ldots $ are independent random variables such that
$\mathbb{P}(Y_{n+1}=1) = p_n = 1 - \mathbb{P}(Y_{n+1}=-1)$. Let $c>0$
and $\sigma_n = \Var(Z_n)^{\frac12}$. The following implications hold
\begin{align}
   \label{eq:limsupgeq}
    p_n
  &=
    \begin{cases}
     0,&\text{if }  n \leq m \\
     \frac{1}{2} - \min\Big\{\frac12,
     \frac{b}{\sqrt{n}}\Big\},&\text{otherwhise}
    \end{cases} 
  \qquad  \Rightarrow  \qquad 
   \Pb\bigg(\limsup_{n} \frac{Z_n}{\sigma_n} \geq c \bigg) =  1, 
   \\[1em]
  \label{eq:limsupgeqCOR}
    p_n
 &=
   \begin{cases}
    1,&\text{if }  n \leq m \\
    \frac{1}{2} + \min\Big\{\frac12,
    \frac{b}{\sqrt{n}}\Big\},&\text{otherwhise}
   \end{cases}
   \qquad  \Rightarrow  \qquad 
   \Pb\bigg (\liminf_{n} \frac{Z_n}{\sigma_n} \leq -c \bigg)
   = 1.
\end{align}
\end{lemma}

\begin{proof}[Proof of Lemma~\ref{lemma:recurrence_pnn}]
We will only present the proof of (\ref{eq:limsupgeq}). The proof of
(\ref{eq:limsupgeqCOR}) goes analogously. Let $A_c = \{\limsup_n
Z_n/\sigma_n \geq c\}$ and $\boocal{Z}_n = \sigma\big(\{Z_k: k \geq
n\}\big)$.  Define $\boocal Z = \bigcap_n \boocal{Z}_n$, the tail
sigma algebra generated by $Z_n$. Observe that, by the definition of
$Z_n$, we have that $\sigma_n = 2\big( \sum_{k=1}^{n}
p_k(1-p_k))^{\frac12}$, thus the event $A_c$ belongs to $\boocal Z$
because $\sigma_n \to\infty$ as $n\to\infty$. Therefore, in order to
show that $\Pb(A_c) = 1$, by Kolmogorov's zero-one law, it suffices to
show that $\Pb(A_c) > 0$.

Since $A_c = \{\limsup_n Z_n/\sigma_n \geq c\} \supseteq \lim\sup_n
\{Z_n/\sigma_n \geq c\}$, we have that
\begin{equation}
 \label{eqn:lowerbound}
 \begin{aligned}
  \Pb(A_c)
  &\geq
  \Pb\bigg(\limsup_n \Big\{\frac{Z_n}{\sigma_n} \geq c\Big\}\bigg) \\
  &\geq
  \limsup_n \Pb\bigg(\frac{Z_n}{\sigma_n} \geq c\bigg) \\
  &=
  \limsup_n \Pb\bigg(\frac{Z_n-\E[Z_n]}{\sigma_n} \geq c -
  \frac{\E[Z_n]}{\sigma_n}\bigg).
 \end{aligned}
\end{equation}

Let $Z$ be a standard normal random variable and let
$\overset{d}{\longrightarrow}$ stand for convergence in
distribution. Let us assume and argue later on that $\ell =
\lim_{n\to\infty} \mathbb{E}(Z_n)/\sigma_n$ exists and that $|\ell| <
\infty$.  By observing that the random variables $Y_n$ are uniformly
bounded and that $\sigma_n = \Var(Z_n)^{\frac12} \to \infty$ as
$n\to\infty$, it follows that Lindeberg's conditions are met and thus
the following central limit theorem holds
\[
   \frac{Z_n - \mathbb{E}(Z_n)}{\sigma_n} 
   \overset{d}{\longrightarrow}Z.
\]
Combining the limit $\ell$ with the bound in (\ref{eqn:lowerbound})
gives
\[
  \Pb(A_c) 
  \geq
 \Pb\big(Z > c - \ell\big) > 0.
\]

To conclude the proof, it remains to verify that the limit $\ell$
exists and is finite. By definition of $Z_n$, it follows that
\begin{equation}\label{conclude}
  \frac{\mathbb{E}(Z_n)}{\sigma_n} 
  =
  \frac{Z_0/2 + \sum_{k=1}^{n} p_k - n/2}
  {\sqrt{\sum_{k=1}^{n} p_k(1-p_k)}},
\end{equation}
where $p_k = \frac{1}{2}- b/\sqrt{k}$ for
sufficiently large $k$. A straightforward computation shows that the
right hand-side of (\ref{conclude}) converges to $-4b$ as $n$
goes to infinity.
\end{proof}

We are ready for the proof of Theorem~\ref{th:rec}.

\begin{proof}[Proof of Theorem~\ref{th:rec}]
Throughout the proof, $i \in \{1,2\}$ will be fixed. It is sufficient
to show that $\Pb\big(\limsup_{n} \{S^i_n = 0\} \big) = 1$.  This will
be achieved by proving that $\Pb\big(\limsup_{n} \{S^i_n = 0\} \big) >
1- \epsilon$ for arbitrary $\epsilon > 0$.  Recall that
$\pi^i_\rstp(X(n))$ is the probability of the event $\{S^i_{n+1} =
S^i_{n} + 1\}$ given $X(n)$ for $n \geq n_0$, where $n_0$ is as
defined in the Introduction. Let $\big\{U_n; n\geq 0 \big\}$ be a
sequence of independent and identically distributed uniform random
variables taking values on the open interval $(0, 1)$ and couple
$S^i_n$ with $U_n$ such that
\[
  S^i_{n+1} = S^i_{n} + 1 \quad\text{if and only if}\quad
  U_{n} \leq P_n \quad \text{ for } \,\, n \geq
  n_0, 
\]
where $P_n = \pi^i_\rstp(X(n))$.

Choose $\varepsilon > 0$ arbitrarily. In accordance to
Corollary~\ref{cor:Aepsilon}, choose $b > 0$ and $m > n_0$ such that
\begin{equation}\label{eq:Aoneminusepsilon}
\Pb( A ) > 1 - \varepsilon, \, \text{ where } \, A =
\bigg \{ \Big | 
P_n - \frac{1}{2} \Big | \leq \frac{b}{\sqrt{n}} \,\,\, \text{for
  all}  \,\,\,  n > m \bigg\}.
\end{equation}

Now, let $Z_n$ be another walk with independent increments such that 
$Z_0 = S^i_0$ and
\[
  Z_{n+1} = Z_n + 1 \,\,\, \text{ if and only if }
\,\,\, U_n  \leq 
  p_n \,\,\, \mbox{ for all } \,\,\,  n \geq 0,
\]
where
\[
  p_n =
  \begin{cases}
    0, &\text{if } 0 \leq n \leq m,              \\ 
    \frac{1}{2} -  \min\Big\{\frac12, \frac{b}{\sqrt{n}}\Big\},
    &\text{if } n > m.
  \end{cases}
\]

Observe that the walks $S^i_n$ and $Z_n$ are coupled through $U_n$ as
follows. Given that $p_n \leq P_n$, it follows that $Z_{n+1} = Z_n +
1$ implies that $S^i_{n+1} = S^i_{n} + 1$. Indeed, given that
$Z_{n+1} = Z_n + 1$ and  $p_n \leq P_n$, we have that $U_n \leq p_n 
\leq P_n$ and therefore $S^i_{n+1} = S^i_{n} + 1$. Since $Z_0 = S^i_0$ 
and $p_n = 0$ for all $n = 0, 1, 2, \ldots, m$, it follows that $S^i_n 
\geq Z_n$ for all $n \geq 0$, given the event $B = \big\{p_n \leq P_n$ 
for all $n > m \big\}$. As a consequence, we have that
\begin{equation}\label{eqn:limsupB}
\begin{aligned}
  \Pb\bigg(\limsup_{n} \frac{S^i_n}{\sigma_n} \geq  c\bigg) 
 &\geq
  \Pb\bigg(\limsup_{n} \frac{S^i_n}{\sigma_n} \geq c\, \bigg|\,
  B\bigg)\Pb(B)  
     \\ 
 &\geq
   \Pb\bigg(\limsup_{n} \frac{Z_n}{\sigma_n} \geq c\, \bigg|\, B 
   \bigg)\Pb(B)  \\
 &=
  \Pb(B), 
\end{aligned}
\end{equation}
where $\sigma_n = \Var(Z_n)^{\frac12}$. The second inequality in
\eqref{eqn:limsupB} follows by the coupling of $S_n$ and $Z_n$. The
equality in \eqref{eqn:limsupB} follows by (\ref{eq:limsupgeq}),
because $Z_n$ satisfies the hypotheses of
Lemma~\ref{lemma:recurrence_pnn}, and hence $\Pb\big(\limsup_n
Z_n/\sigma_n \geq c\big ) = 1$ as well as $\Pb\big(\limsup_n
Z_n/\sigma_n \geq c \, | \, B \big ) = 1$.

Now, using (\ref{eqn:limsupB}) and observing that $B \supseteq A$, 
we have that
\begin{equation}\label{eqn:lssigmanome}
 \Pb\bigg(\limsup_{n} \frac{S^i_n}{\sigma_n} \geq  c\bigg)  \geq
 \Pb(B) \geq  \Pb(A) \geq 1-\varepsilon,  
\end{equation}
where the last inequality in (\ref{eqn:lssigmanome}) follows by
definition of $A$ in (\ref{eq:Aoneminusepsilon}).

Since $\varepsilon$ and $c > 0$ where arbitrarily chosen, we have,
by (\ref{eqn:lssigmanome}), that $\Pb\big(\limsup_{n} S^i_n/\sigma_n
\geq c\big) = 1$ for all $c > 0$. By using \eqref{eq:limsupgeqCOR}, we
can show analogously that that $\Pb\big(\liminf_{n} S^i_n/\sigma_n
\leq -c\big) = 1$ for all $c > 0$. Using these two facts, and taking
into account that $\sigma_n$ converges to infinity as $n \to \infty$,
we conclude that
\begin{align*}
    \mathbb{P}\Big(\limsup_{n} \{S_n^i = 0\}\Big)
    &\geq
    \mathbb{P}\bigg(\limsup_{n} \frac{S_n^i}{\sigma_n} = +\infty,\ \
    \liminf_{n}  \frac{S_n^i}{\sigma_n} = - \infty \bigg) \\
    &=
    \lim_{c \rightarrow \infty} \mathbb{P}\bigg(\limsup_{n}
    \frac{S^i_n}{\sigma_n} \geq  c,\ \ \liminf_{n} \frac{S_n^i}{\sigma_n}
      \leq - c\bigg) = 1.
    \qedhere
\end{align*}
\end{proof}

\begin{remark}\label{remark:speed}
The conclusion of Lemma~\ref{lemma:recurrence_pnn}, required in the
demonstration of Theorem~\ref{th:rec}, relies
on the fact that $p_n$ converges sufficiently fast towards
$\frac{1}{2}$. The computations involved in
Lemma~\ref{lemma:recurrence_pnn} show that the rate of convergence
must be such that $p_n \leq \frac{1}{2} - b n^{-\rho}$ for $b > 0$
and $\rho = \frac12$ for suficiently large $n$.  The exponent $\rho
= \frac12$ is critical in the sense that the conclusion of
Lemma~\ref{lemma:recurrence_pnn} does not holds if $\rho <
\frac12$.  According to Lemma~\ref{lem:speedconv}, when $\beta \in
[0,1]$ we have exactly the critical rate $\rho = \frac12$. The same
arguments used throughout the proof of Lemma~\ref{lem:speedconv} also
give the estimate $\rho = \zeta$, for $\zeta = 1 - \beta/2$ when
$\beta \in (1, 2]$. This shows that the convergence of $p_n$ towards
$\frac12$ can be arbitrarily slow as $\beta \nearrow 2$, and in fact
too slow for any $\beta>1$. The question about the 
recurrence/transience of both random walks remains therefore open when 
$\beta \in (1, 2]$.
\end{remark}

\section*{Appendix}
\begin{proof}[Proof of Lemma~\ref{lem:X_is_SA}] By (\ref{eqn:xxis}),
  it follows that
\begin{align*}
  X^i_{\lstp}(n+1) - X^i_{\lstp}(n) 
  &=
    \frac{1}{n+1}\big(-X^i_\lstp(n) + \xi^i_\lstp(n)\big).
\end{align*}

Likewise, an analogous expression for $X^i_\rstp(n+1) - X^i_\rstp(n)$
can be derived in terms of $\xi^i_\rstp(n)$ and $X^i_\rstp(n)$.
Hence, by using (\ref{eqn:gamma_and_U}) and (\ref{eqn:TheField}), it
follows that  
\begin{equation}\label{eqn:increment}
  X(n+1) - X(n)
  =
  \gamma_n\Big\{ F\big (X(n) \big ) + \E[\xi(n)\mid\sF_n] - \pi\big
  (X(n) \big )  +   U_n\Big\}. 
\end{equation}
	
To conclude that (\ref{eqn:SA}) holds, we will show that
$\E[\xi(n)\mid\sF_n] - \pi(X(n)) = \mathbf{0}$. By using the
definition of the probabilities \eqref{eqn:transProb} and of $\xi(n)$
in (\ref{eqn:xis}), and further observing that $\psi$, defined in
(\ref{eqn:psi}), satisfies $1 - \psi(y) = \psi(-y)$ for all $y$, we
have that
\begin{align*}
  \E[\xi(n)\mid\sF_n]
  &=
    \big(
     \Pb(S^1_{n+1}-S^1_n=-1\mid\sF_n), \ldots, 
     \Pb(S^2_{n+1}-S^2_n= 1\mid\sF_n)
    \big)  \\
  &=
    \Big(
     \psi \Big(\frac{S^2_{0}-S^2_n}{n}  \Big),
     \psi \Big( \frac{S^2_{n}-S^2_0}{n}  \Big),
     \psi \Big(\frac{S^1_{0}-S^1_n}{n}  \Big),
     \psi \Big( \frac{S^1_{n}-S^1_0}{n}  \Big)
    \Big).  
\end{align*}
Now, since $(S^j_{0}-S^j_n)/n = 2 X_\lstp^j(n) - 1$ and
$(S^j_{n}-S^j_0)/n = 2 X_\rstp^j(n) - 1$, we conclude, by using the
definition of $\pi$, given by (\ref{eqn:pifirst}) and
(\ref{eqn:pipsi}), that
\begin{equation}\label{eqn:epi}
  \E[\xi(n)\mid\sF_n] =
  \big(
  \pi_\lstp^1 \big(X(n)\big),
  \pi_\rstp^1 \big(X(n)\big),
  \pi_\lstp^2 \big(X(n)\big),
  \pi_\rstp^2 \big(X(n)\big)\big) =
  \pi \big(X(n)\big).
\end{equation}
In view of (\ref{eqn:epi}), equation (\ref{eqn:increment})
reduces to (\ref{eqn:SA}).  
\end{proof}

\begin{proof}[Proof of Lemma~\ref{lem:Kushner}.\ref{as:KushnerLemma}]
Let $M_n = \sum_{k=0}^n \gamma_k U_k$. The process $\{M_n\}_{n\geq 0}$
is a martingale with respect to $\{\sF_n\}_{n\geq 0}$, that is
\[
   \E[M_{n+1}\mid\sF_{n+1}] = \sum_{k=0}^n \gamma_k U_k + 
   \gamma_{n+1}\E[U_{n+1}\mid\sF_{n+1}] = M_n.
\]
Observe that 
\begin{align*}
   \E\big[\|M_{n+1} - M_n\|^2\big|\ \sF_{n+1}\big] 
   &= 
   \gamma_{n+1}^2 \E\big[\|U_{n+1}\|^2\big|\ \sF_{n+1}\big]  \\
   &\leq 
   \gamma_{n+1}^2 \bigg(\sum_{v \in \{\lstp, \rstp\},\ i \in
       \{1,2\}} \xi^i_v(n+1) \bigg)^2
   \leq 
   16\, \gamma^2_{n+1}.
\end{align*}
By using Doob's decomposition for the sub-martingale $M_n^2$, consider
the predictable increasing sequence $A_{n+1} = M_n^2 + M_n$ with
$A_1=0$. The conditional variance formula for the increment $M_{n+1} -
M_n$ gives
\[
    A_{n+2}-A_{n+1} 
	= 
	\E\big[M_{n+1}^2\big|\ \sF_n\big] - M_n^2 =
        \E[\|M_{n+1}-M_n\|^2\mid  \sF_{n+1}],
\]
and hence for any $n$,
\[
   A_{n+2}
  = 
   \sum_{k=0}^n \E\big[\|M_{k+1} 
   - M_k\|^2\big|\ \sF_{n+1}\big]
  \leq
    16 \sum_{k=0}^n \gamma_{n+1}^2.
\]
Passing to the limit $n\to\infty$ shows that almost surely $A_\infty <
\infty$. According to Theorem~5.4.9 in \cite{D2010}, p. 254, this in
turn implies that $M_n$ converges almost surely to a finite limit in
$\R^{2\times2}$ and hence that $\{M_n\}_{n\geq 0}$ is a Cauchy
sequence. This is sufficient to conclude the proof.
\end{proof}

\bibliographystyle{alpha}

\vspace{24pt}

\end{document}